\documentclass[oneside,english]{amsart}
\usepackage[T1]{fontenc}
\usepackage[latin9]{inputenc}
\usepackage{geometry}
\usepackage{amstext}
\usepackage{amsthm}
\usepackage{amssymb}
\usepackage{esint}

\makeatletter
\numberwithin{equation}{section}
\numberwithin{figure}{section}
\theoremstyle{plain}
\newtheorem{thm}{\protect\theoremname}
\theoremstyle{plain}
\newtheorem{lem}[thm]{\protect\lemmaname}
\theoremstyle{remark}
\newtheorem{rem}[thm]{\protect\remarkname}
\theoremstyle{definition}
\newtheorem{problem}[thm]{\protect\problemname}

\usepackage{todonotes}
\global\long\def\Re{\operatorname{Re}}

\global\long\def\Im{\operatorname{Im}}

\global\long\def\Arg{\operatorname{Arg}}

\makeatother

\usepackage{babel}
\providecommand{\lemmaname}{Lemma}
\providecommand{\problemname}{Problem}
\providecommand{\remarkname}{Remark}
\providecommand{\theoremname}{Theorem}

\begin{document}

\title{Zeros of polynomials with four-term recurrence and linear coefficients}

\author{Khang Tran \and Andres Zumba\\
California State University, Fresno}
\begin{abstract}
This paper investigates the zero distribution of a sequence of polynomials
$\left\{ P_{m}(z)\right\} _{m=0}^{\infty}$ generated by the reciprocal
of $1+ct+B(z)t^{2}+A(z)t^{3}$ where $c\in\mathbb{R}$ and $A(z)$,
$B(z)$ are real linear polynomials. We study necessary and sufficient
conditions for the reality of the zeros of $P_{m}(z)$. Under these
conditions, we find an explicit interval containing these zeros, whose
union forms a dense subset of this interval. 
\end{abstract}

\maketitle

\section{Introduction}

The understanding of zeros of polynomials defined recursively plays
an important part in the study of zero distribution of polynomials.
A classic recurrence is the three-term recursive formula since it
is a necessary condition for a sequence of polynomials to be orthogonal.
Orthogonality, in turn, establishes the reality of the zeros of the
sequence of polynomials. 

Much less is known about the zero distribution of a sequence of polynomials
$\left\{ P_{m}(z)\right\} _{m=0}^{\infty}$ satisfying a four-term
recurrence 
\begin{equation}
P_{m}(z)+C(z)P_{m-1}(z)+B(z)P_{m-2}(z)+A(z)P_{m-3}(z)=0,\qquad A(z),B(z),C(z)\in\mathbb{\mathbb{R}}[z],\label{eq:fourtermrecurrence}
\end{equation}
for $m\ge1$, with initial conditions $P_{0}(z)\equiv1$ and $P_{-m}(z)\equiv0$.
Equivalently, this sequence of polynomials is generated by 
\begin{equation}
\sum_{m=0}^{\infty}P_{m}(z)t^{m}=\frac{1}{1+C(z)t+B(z)t^{2}+A(z)t^{3}}.\label{eq:generalgen}
\end{equation}
For the zero distribution of a special four-term recurrence, see \cite{ghr}.
Due to its four-term recurrence form, this sequence may not be orthogonal
and the reality of the zeros of $P_{m}(z)$ for certain coefficient
polynomials $A(z)$, $B(z)$, and $C(z)$ is not immediate. Even when
$A(z)$, $B(z)$, and $C(z)$ are linear polynomials in $z$, the
conditions for the reality of these zeros are still unknown. In \cite[Theorem 1]{orr},
the authors found a sufficient condition for the reality of the zeros
of $P_{m}(z)$ when $C(z)=z$ and $A(z)$ and $B(z)$ are constant.
This condition was later shown to be necessary in \cite[Proposition 1]{bbs}.
For the case when $A(z)$ is linear and $B(z)$ and $C(z)$ are constant,
such a condition was established in \cite{tz}.

It is natural to consider the reality of the zeros of $P_{m}(z)$
when only one of the coefficient polynomials $A(z)$, $B(z)$, and
$C(z)$ is constant. The goal of this paper is to find necessary and
sufficient conditions for the reality of the zeros of $P_{m}(z)$
when $C(z)\equiv c$, $B(z)=b_{0}+b_{1}z$, and $A(z)=a_{0}+a_{1}z$
where $c,b_{0},b_{1},a_{0},a_{1}\in\mathbb{R}$, $ca_{1}b_{1}\le0$,
and $ca_{1}\ne0$. The case $ca_{1}b_{1}>0$ remains unknown to the
authors. 
\begin{thm}
\label{thm:generalthm}Suppose the sequence $\left\{ P_{m}(z)\right\} $
is defined as above where $ca_{1}b_{1}\le0$ and $ca_{1}\ne0$. The
zeros of $P_{m}(z)$ are real for all $m\in\mathbb{N}$ if and only
if 
\begin{equation}
1+a+b\ge0\qquad\text{and}\qquad9-27a+b\ge0\label{eq:cond}
\end{equation}
where 
\begin{align*}
a & :=\frac{b_{0}}{c^{2}}-\frac{b_{1}a_{0}}{c^{2}a_{1}},\\
b & :=-\frac{b_{1}c}{a_{1}}.
\end{align*}
\end{thm}

Under \eqref{eq:cond}, we can find an explicit real interval containing
the zeros of $P_{m}(z)$ by considering the cubic polynomial 
\[
(8a-2)\zeta^{3}+\zeta^{2}(-12a+b+5)+(6a-2)\zeta-a
\]
whose only real zero on $(-\infty,-1]\cup[1,\infty)$ is denoted by
$\zeta_{0}$. The existence and uniqueness of such zero is justified
in Section \ref{sec:auxiliary-functions}. The zeros of $P_{m}(z)$
lie on the interval
\begin{equation}
\frac{c^{3}}{a_{1}}I_{a,b}-\frac{a_{0}}{a_{1}}\label{eq:intPm}
\end{equation}
where
\begin{equation}
I_{a,b}=\left(-\infty,\frac{\zeta_{0}^{2}}{(1-2\zeta_{0})^{3}}\right].\label{eq:intdef}
\end{equation}
Moreover, if we let $\mathcal{Z}(P_{m})$ be the set of zeros of $P_{m}(z)$,
then $\bigcup_{m=0}^{\infty}\mathcal{Z}(P_{m})$ is dense on \eqref{eq:intPm}.
In the special case $b_{1}=0$, $a_{1}=1$, and $a_{0}=0$, we solve
\[
\zeta_{0}=\frac{2a-1-\sqrt{1-3a}}{4a-1}
\]
and obtain Case (ii) of Theorem 1 in \cite{tz}.

Our approach to the proof of Theorem \ref{thm:generalthm} relies
on the reparametrization from $P_{m}(z)$ to $P_{m}(z(\theta))$ where
$z(\theta)$ is strictly monotone. This function $z(\theta)$ is constructed
by an auxiliary function $\zeta(\theta)$ which is defined implicitly
through the bivariate function $f(\zeta,\theta)$ (c.f. \eqref{eq:zetapoly}).
We count the number of zeros in $\theta$ of $P_{m}(z(\theta))$,
each of which yields a distinct real zero of $P_{m}(z)$ by the monotonicity
of $z(\theta)$. If the number of counted zeros is the same as the
degree of $P_{m}(z)$, then all the zeros of $P_{m}(z)$ are real
by the Fundamental Theorem of Algebra. Our paper is organized as follows.
Section \ref{sec:auxiliary-functions} studies the auxiliary function
$\zeta(\theta)$ and Section \ref{sec:zthetamonotone} establishes
the monotone property of $z(\theta)$. With all the properties in
these two sections, we prove the sufficient and necessary condition
for the reality of the zeros of $P_{m}(z)$ in Sections 4 and 5 respectively.

\section{\label{sec:auxiliary-functions}Auxiliary functions }

Our first step is to simplify the right side of \eqref{eq:generalgen}.
We note that the substitutions $t\rightarrow t/c$ and 
\[
\frac{a_{1}}{c^{3}}z+\frac{a_{0}}{c^{3}}\rightarrow z
\]
reduce the right side of \eqref{eq:generalgen} to 
\[
\frac{1}{1+t+at^{2}+zt^{2}(t-b)}=:\frac{1}{D(t,z)}.
\]
We deduce that Theorem \ref{thm:generalthm} is equivalent to the
following theorem.
\begin{thm}
\label{thm:maintheorem}Suppose $b\ge0$. The zeros of $H_{m}(z)$
generated by 
\begin{equation}
\sum_{m=0}^{\infty}H_{m}(z)t^{m}=\frac{1}{1+t+at^{2}+zt^{2}(t-b)}:=\frac{1}{D(t,z)}\label{eq:genfunc}
\end{equation}
are real if and only if 
\[
1+a+b\ge0\qquad\text{and}\qquad9-27a+b\ge0.
\]
\end{thm}

Since the case $b=0$ is proved in Theorem 2 of \cite{tz}, we only
consider $b>0$ in this paper. In fact, to prove the sufficient condition
for the reality of the zeros of $H_{m}(z)$, for each $b>0$, we can
ignore certain values of $a$ by the lemma below. 
\begin{lem}
\label{lem:densesubint}We fix $b>0$ and let $S$ be a dense subset
of $[-1-b,(b+9)/27]$. If 
\[
\mathcal{Z}(H_{m}(z,a,b))\subset I_{a,b}
\]
for all $a\in S$, then 
\[
\mathcal{Z}(H_{m}(z,a^{*},b))\subset I_{a^{*},b}
\]
for all $a^{*}\in[-1-b,(b+9)/27]$. 
\end{lem}

\begin{proof}
Let $a^{*}\in[-1-b,(b+9)/27]$ be given. By the density of $S$ in
$[-1-b,(b+9)/27]$, we can find a sequence $\{a_{n}\}$ in $S$ such
that $a_{n}\rightarrow a^{*}$. For any $z^{*}\notin I_{a^{*},b}$,
we will show that $H_{m}(z^{*},a^{*},b)\ne0$. We note that the zeros
of $H_{m}(z,a_{n},b)$ lie in the interval $I_{a_{n},b}$ whose right
endpoint approaches the right endpoint of $I_{a^{*},b}$ as $n\rightarrow\infty$.
If we let $z_{k}^{(n)}$, $1\le k\le\deg H_{m}(z,a_{n},b)$, be the
zeros of $H_{m}(z,a_{n},b)$ then 
\[
\left|H_{m}(z^{*},a_{n},b)\right|=\gamma^{(n)}\prod_{k=1}^{\deg H_{m}(z,a_{n},b)}\left|z^{*}-z_{k}^{(n)}\right|
\]
where $\gamma^{(n)}$ is the leading coefficient of $H_{m}(z,a_{n},b)$.
Since $\deg H_{m}(z,a_{n})\le\left\lfloor m/2\right\rfloor $ by Lemma
\ref{lem:degreeHm}, using this product representation and the assumption
that $z^{*}\notin I_{a,b}$, we conclude that there is a fixed (independent
of $n$) $\delta>0$ so that $|H_{m}(z^{*},a_{n},b)|>\delta$, for
all large $n$. Since $H_{m}(z^{*},a,b)$ is a polynomial in $a$
for any fixed $z^{*}$, we conclude that 
\[
H_{m}(z^{*},a^{*},b)=\lim_{n\rightarrow\infty}H_{m}(z^{*},a_{n},b)\ne0
\]
and the result follows. 
\end{proof}
As suggested in the introduction, we will count the number of real
zeros of $H_{m}(z)$ and compare this number to its degree. The lemma
below provides an upper bound for the degree. 
\begin{lem}
\label{lem:degreeHm}The degree of the polynomial $H_{m}(z)$ defined
by \eqref{eq:genfunc} is at most $\left\lfloor m/2\right\rfloor $. 
\end{lem}

\begin{proof}
This lemma follows easily from induction applied to the recurrence
\[
H_{m}(z)+H_{m-1}(z)+(a-bz)H_{m-2}(z)+zH_{m-3}(z)=0,\qquad m\ge1,
\]
and the initial condition $H_{0}(z)\equiv1$ and $H_{m}(z)\equiv0$
for $m<0$. 
\end{proof}
To motivate the formula for the function $z(\theta)$ mentioned in
the introduction, we provide some heuristic arguments. For each $z\in\mathbb{R}\backslash\{0\}$,
we let $t_{0}=t_{0}(z),$ $t_{1}=t_{1}(z)$, and $t_{2}=t_{2}(z)$
be the three zeros of $D(t,z)$. If $t_{0}$ and $t_{1}$ are two
distinct complex conjugates and $t_{2}\in\mathbb{R}$, then we let
$t_{0}=\tau e^{-i\theta}$, $t_{1}=\tau e^{i\theta}$, and $t_{2}=\zeta\tau$
where $\zeta\in\mathbb{R}$. From the elementary symmetric equations
\begin{equation}
t_{0}+t_{1}+t_{2}=\frac{bz-a}{z},\qquad t_{0}t_{1}+t_{0}t_{2}+t_{1}t_{2}=\frac{1}{z},\qquad\text{and \qquad}t_{0}t_{1}t_{2}=-\frac{1}{z},\label{eq:elemsym}
\end{equation}
we deduce that 
\[
1+e^{2i\theta}+\zeta e^{i\theta}=\frac{bz-a}{zt_{0}},\qquad e^{2i\theta}+\zeta e^{i\theta}+\zeta e^{3i\theta}=\frac{1}{zt_{0}^{2}},\qquad\text{and}\qquad\zeta e^{3i\theta}=-\frac{1}{zt_{0}^{3}}.
\]
We divide the first equation by $e^{i\theta}$, the second by $e^{2i\theta}$,
and the third by $e^{3i\theta}$ and obtain 
\begin{equation}
2\cos\theta+\zeta=\frac{bz-a}{z\tau},\qquad1+2\zeta\cos\theta=\frac{1}{z\tau^{2}},\qquad\text{and}\qquad\zeta=-\frac{1}{z\tau^{3}}.\label{eq:elemsymzetatheta}
\end{equation}
We solve for $z$ from the third equation 
\begin{equation}
z=-1/\zeta\tau^{3}\label{eq:zdef}
\end{equation}
and substitute $z$ to the first equation 
\begin{equation}
\tau(2\cos\theta+\zeta)=b+a\zeta\tau^{3}\label{eq:tauequation}
\end{equation}
and the second equation

\begin{equation}
\tau=-\frac{1}{\zeta}-2\cos\theta.\label{eq:taudef}
\end{equation}
From these identities, we obtain 
\begin{equation}
\zeta(1+2\zeta\cos\theta)(2\cos\theta+\zeta)=-b\zeta^{2}+a(1+2\zeta\cos\theta)^{3}\label{eq:zetatheta}
\end{equation}
which motivates the definition of the function

\begin{equation}
f(\zeta,\theta)=\zeta(1+2\zeta\cos\theta)(2\cos\theta+\zeta)+b\zeta^{2}-a(1+2\zeta\cos\theta)^{3}.\label{eq:zetapoly}
\end{equation}
Converse to the construction above, we have the following lemma. 
\begin{lem}
\label{lem:zerosdenom}For any $\theta\in(0,\pi)$, if $\zeta$ is
a zero of $f(\zeta,\theta)$ and $z$ and $\tau$ are given in \eqref{eq:zdef}
and \eqref{eq:taudef}, then $\tau e^{\pm i\theta}$ and $\zeta\tau$
are the three zeros of $D(t,z)$. 
\end{lem}

\begin{proof}
We reverse the arguments above by combining \eqref{eq:zetatheta}
and \eqref{eq:taudef} to obtain \eqref{eq:tauequation}. Together
with \eqref{eq:zdef}, we deduce \eqref{eq:elemsymzetatheta} and
\eqref{eq:elemsym} follows. 
\end{proof}
As a polynomial in $\zeta$, its reciprocal $f^{*}(\zeta,\theta):=\zeta^{3}f(1/\zeta,\theta)$
is 
\begin{align}
f^{*}(\zeta,\theta) & =(\zeta+2\cos\theta)(2\zeta\cos\theta+1)+b\zeta-a(\zeta+2\cos\theta)^{3}\nonumber \\
 & =-a\zeta^{3}+(2\cos\theta-6a\cos\theta)\zeta^{2}+(1+b+4\cos^{2}\theta-12a\cos^{2}\theta)\zeta+2\cos\theta-8a\cos^{3}\theta.\label{eq:zetapolyrecip}
\end{align}
For the sufficient direction of Theorem \ref{thm:maintheorem}, we
limit the domain of $\theta$ to $(\pi/2,\pi)$. Our first goal here
is to show that for any $\theta\in(\pi/2,\pi)$, $f^{*}(\zeta,\theta)$
has exactly one real zero on the interval $(-1,1)$ by considering
the sign of this polynomial at the endpoints. 
\begin{lem}
\label{lem:oppsigns}For any fixed $\theta\in(\pi/2,\pi)$, if 
\begin{align*}
2-8a+8a^{2}+ab & \ne0,\\
b+1-a & \ne0,\\
9-27a+b & \ne0,\\
1+a+b & \ne0,
\end{align*}
then we have 
\[
f^{*}(-1,\theta)f^{*}(1,\theta)<0.
\]
\end{lem}

\begin{proof}
If we let $x=\cos\theta$, then $f^{*}(-1,\theta)f^{*}(1,\theta)$
is a cubic polynomial in terms of $x^{2}$. We let that polynomial
be $g(x)$ and by a computer its discriminant is 
\[
-4096b(27a^{2}b+4)(2-8a+8a^{2}+ab)^{2}<0,
\]
from which we deduce that $g(x)$ has only one real root. Then the
two inequalities 
\begin{align*}
g(0) & =-(b+1-a)^{2}<0,\\
g(1) & =-(9-27a+b)(1+a+b)<0,
\end{align*}
imply that $g(x)<0$, $\forall x\in(0,1)$ and the lemma follows. 
\end{proof}
We note that Lemma \ref{lem:densesubint} allows us to focus on the
values of $a$ in which the all conditions of Lemma \ref{lem:oppsigns}
are met. In fact, when $a>0$ we know the sign of each factor $f^{*}(-1,\theta)$
and $f^{*}(1,\theta)$ in the lemma below.
\begin{lem}
\label{lem:signf*atpm1} If $a>0$ and $b-27a+9>0$, then $f^{*}(1,\theta)>0$
and $f^{*}(-1,\theta)<0$ for all $\theta\in(\pi/2,\pi)$. 
\end{lem}

\begin{proof}
To show $f^{*}(1,\theta)>0$, we consider two cases $0<a\le1$ and
$a>1$. In the first case, the inequalities 
\[
a(2\cos\theta+1)-1<a-1<0
\]
imply 
\begin{align*}
-f^{*}(1,\theta) & =a-(2\cos\theta-6a\cos\theta)-(1+b+4\cos^{2}\theta-12a\cos^{2}\theta)-(2\cos\theta-8a\cos^{3}\theta)\\
 & =(a+6a\cos\theta+12a\cos^{2}\theta+8a\cos^{3}\theta)-(4\cos^{2}\theta+4\cos\theta+1)-b\\
 & =a(1+2\cos\theta)^{3}-(1+2\cos\theta)^{2}-b\\
 & =(2\cos\theta+1)^{2}(a(2\cos\theta+1)-1)-b<0.
\end{align*}
In the later case, we have 
\begin{align*}
-f^{*}(1,\theta) & =a-2\cos\theta(2-3a)-(1+b+4\cos^{2}\theta-12a\cos^{2}\theta)+8a\cos^{3}\theta\\
 & <a-2\cos\theta(2-3a)-1-b-4\cos^{2}\theta+12a+8a\cos^{3}\theta\\
 & =-2\cos\theta(2-3a)-1-4\cos^{2}\theta+8a\cos^{3}\theta+\left(-b+27a-9\right)+\left(-14a+9\right)<0.
\end{align*}
The claim that $f^{*}(-1,\theta)<0$ follows from Lemma \ref{lem:oppsigns}. 
\end{proof}
\begin{rem}
\label{rem:zerosf*a>0}As a consequence of Lemma \ref{lem:signf*atpm1}
and the fact that the leading coefficient of $f^{*}(\zeta,\theta)$
is $-a$, we conclude that if $a>0$, then $f^{*}(\zeta,\theta)$
has one zero on each of the interval $(-\infty,-1)$, $(-1,1)$, and
$(1,\infty)$ and consequently this polynomial has exactly one zero
on $(-1,1)$.

For the case $a<0$, we consider the lemma below. 
\end{rem}

\begin{lem}
\label{lem: a<0 roots f*} Assume $a<0$, $1+b+a>0$, and $\theta\in(\pi/2,\pi)$.
If all the zeros in $\zeta$ of $f^{*}(\zeta,\theta)$ are real, then
exactly one of them lies in the interval $(0,1)$, and the other two
lie in $(1,\infty)$. 
\end{lem}

\begin{proof}
By Lemma \ref{lem:oppsigns} and \eqref{eq:zetapolyrecip}, the real
zeros of $f^{*}(\zeta,\theta)$ are positive and at least one of which
lie in $(0,1)$. If all the zeros of $f^{*}(\zeta,\theta)$ are real,
then so are two zeros of its derivative $df^{*}(\zeta,\theta)/d\zeta$
\[
\frac{(2-6a)\cos\theta\pm\sqrt{3a\left(b-4\cos^{2}\theta+1\right)+4\cos^{2}\theta}}{3a}.
\]
With the note that the leading coefficient of $df^{*}(\zeta,\theta)/d\zeta$
is positive, we will show that these two zeros lie in the interval
$(1,\infty)$ by claiming that one of the two zeros lies in this interval
and $df^{*}(\zeta,\theta)/d\zeta>0$ when $\zeta=1$. The lemma will
follow from the interlacing zeros of $f^{*}(\zeta,\theta)$ and its
derivative. The second claim comes directly from the identity 
\[
\left.\frac{df^{*}(\zeta,\theta)}{d\zeta}\right|_{\zeta=1}=(2\cos\theta+1)^{2}(1-3a)+b>0.
\]

Since the two zeros of $df^{*}(\zeta,\theta)/d\zeta$ are real, we
have 
\[
4\cos^{2}\theta+3a+3ab-12a\cos^{2}\theta\geq0
\]
or equivalently 
\begin{equation}
\cos\theta\leq-\sqrt{\frac{-3a(1+b)}{4(1-3a)}}.\label{eq:upperboundcosine}
\end{equation}
On the other hand, the assumption $a<0$ implies that $1+b-3ab>0$
and consequently 
\[
\sqrt{1+b}>\sqrt{\frac{-3a}{1-3a}}.
\]
We multiply both sides of this inequality by $\sqrt{-3a/4(1-3a)}$
and apply \eqref{eq:upperboundcosine} to get 
\[
\cos\theta\leq-\sqrt{\frac{-3a(1+b)}{4(1-3a)}}<\frac{3a}{2(1-3a)}
\]
which gives
\[
2(1-3a)\cos\theta-3a<0.
\]
Hence 
\[
2(1-3a)\cos\theta-3a<\sqrt{4\cos^{2}\theta+3a+3ab-12a\cos^{2}\theta}
\]
and consequently
\[
\frac{2(1-3a)\cos\theta-\sqrt{4\cos^{2}\theta+3a+3ab-12a\cos^{2}\theta}}{3a}>1.
\]
\end{proof}
With all the previous lemmas at our disposal, the formal proof of
Theorem \ref{thm:maintheorem} begins by the definition of the function
$1/\zeta(\theta)$ as the only real zero of $f^{*}(\zeta,\theta)$
on the interval $(-1,1)$. The existence and uniqueness of this zero
comes from Remark \ref{rem:zerosf*a>0} and Lemma \ref{lem: a<0 roots f*}.
By the Implicit Function Theorem $1/\zeta(\theta)$ is smooth on $(\pi/2,\pi)$.
We next define the two functions $\tau(\theta)$ and $z(\theta)$
according to \eqref{eq:taudef} and \eqref{eq:zdef} respectively.
Since $1/\zeta(\theta)$ is smooth on $(\pi/2,\pi)$, so is $\tau(\theta)$.
With Lemma \ref{lem:positivetau} below and \eqref{eq:zdef}, the
function $z(\theta)$ is also smooth on $(\pi/2,\pi)$. 
\begin{lem}
\label{lem:positivetau}For any $\theta\in(\pi/2,\pi)$, we have $\tau(\theta)>0.$ 
\end{lem}

\begin{proof}
We will show that $\tau(\theta)$ has no zero on $(\pi/2,\pi)$ and
the lemma will follow from 
\[
\tau(2\pi/3)=-\frac{1}{\zeta(2\pi/3)}+1>0.
\]
Indeed, if $\theta_{0}\in(\pi/2,\pi)$ is a zero of $\tau(\theta)$,
then $1/\zeta(\theta_{0})\ne0$ and $1+2\zeta(\theta_{0})\cos\theta_{0}=0$,
a contradiction to \eqref{eq:zetatheta}. 
\end{proof}
\begin{lem}
\label{lem:asymptotes}The only zero of $1/\zeta(\theta)$ on $(\pi/2,\pi)$
is $\cos^{-1}(-1/(2\sqrt{a}))$ when $a>1/4$. 
\end{lem}

\begin{proof}
By the definition of $1/\zeta(\theta)$, we note that $\theta$ is
a zero of $1/\zeta(\theta)$ if and only if the free coefficient of
$f^{*}(\zeta,\theta)$ 
\begin{equation}
2\cos\theta(1-4a\cos^{2}\theta)=0.\label{eq:freecoefficientreciprocal}
\end{equation}
unless $f^{*}(\zeta,\theta)$ is a constant $0$ polynomial under
\eqref{eq:freecoefficientreciprocal}. However, this case does not
occur since the coefficient of $\zeta$ of $f^{*}(\zeta,\theta)$
is 
\[
1+b+4\cos^{2}\theta-12a\cos^{2}\theta
\]
which is nonzero when $1-4a\cos^{2}\theta=0$ because 
\[
b+\frac{1}{a}-2>b+\frac{27}{9+b}-2=\frac{b^{2}+7b+9}{9+b}>0.
\]
\end{proof}
\begin{lem}
\label{lem:signzeta} If $1/4<a<b/27+1/3$ and $-1<\cos\theta<-1/2\sqrt{a}$,
then $\zeta(\theta)<0$. 
\end{lem}

\begin{proof}
From Lemma \ref{lem:asymptotes} , $\zeta(\theta)$ is continuous
on $(\cos^{-1}(-1/(2\sqrt{a})),\pi)$ and does not change its sign
on this interval. Thus it suffices to consider the sign of $\zeta(\theta)$
at a single point. We consider the two cases below.

In the case $a\geq1/3$, we let $\theta\rightarrow\pi$ and observe
from \eqref{eq:zetatheta} that $\zeta(\theta)$ approaches $\zeta_{0}$
where 
\begin{align}
0 & =(-2+8a)\zeta_{0}^{3}+(-12a+b+5)\zeta_{0}^{2}+(-2+6a)\zeta_{0}-a\nonumber \\
 & =(-2+8a)\zeta_{0}^{3}+(-27a+b+9)\zeta_{0}^{2}+(-2+8a)\zeta_{0}+(12a-4)\zeta_{0}^{2}+3a\zeta_{0}^{2}-2a\zeta_{0}-a.\label{eq:fpi}
\end{align}
If by contradiction that $\zeta_{0}>0$, then $\zeta_{0}\ge1$ by
Lemma \ref{lem:oppsigns} and consequently 
\[
3a\zeta_{0}^{2}-2a\zeta_{0}-a=2a\zeta_{0}(\zeta_{0}-1)+a(\zeta_{0}^{2}-1)\ge0.
\]
Under the assumption that $a\ge1/3$, we have $-2+8a>0$ and all other
the coefficients of \eqref{eq:fpi} are nonnegative which is a contradiction.

Similarly, in the case $1/4<a<1/3$, E.q. \eqref{eq:zetatheta} with
$\theta\rightarrow\pi$ yields 
\begin{align*}
0>-b\zeta_{0}^{2} & =(1-2\zeta_{0})(\zeta_{0}(-2+\zeta_{0})-a(1-2\zeta_{0})^{2})\\
 & =(1-2\zeta_{0})(\zeta_{0}^{2}(1-4a)+2\zeta_{0}(-1+2a)-a).
\end{align*}
With the same arguments in the previous case, we conclude that $\zeta_{0}<0$. 
\end{proof}
\begin{lem}
\label{lem:limitzetaatasymptote0}If $9-27a+b\geq0$, then $\zeta(\theta)\rightarrow+\infty$
as $\cos\theta\rightarrow0^{-}$. 
\end{lem}

\begin{proof}
As $\cos\theta\rightarrow0^{-}$, the reciprocal of $f(\zeta,\theta)$
as a polynomial in $\zeta$ approaches 
\[
(b+1)\zeta-a\zeta^{3}
\]
which has a simple zero at $0$. Thus exactly one of the zero in $\zeta$
of $f(\zeta,\theta)$ approaches $\pm\infty$. Since the sum of the
three zeros of $f(\zeta,\theta)$ is 
\[
\frac{1+b+4\cos^{2}\theta-12a\cos^{2}\theta}{8a\cos^{3}\theta-2\cos\theta}\rightarrow+\infty,
\]
as $\cos\theta\rightarrow0^{-}$, we conclude that $\zeta(\theta)\rightarrow+\infty$. 
\end{proof}
In the case $a>1/4$, from Lemmas \ref{lem:asymptotes} and \ref{lem:limitzetaatasymptote0},
the continuity of $\zeta(\theta)$ on $\left(\pi/2,\cos^{-1}(-1/2\sqrt{a})\right)$,
and the inequality $\left|\zeta(\theta)\right|>1$, we deduce that
$\zeta(\theta)\rightarrow+\infty$ as $\cos\theta\rightarrow-1/2\sqrt{a}^{+}$.

\section{\label{sec:zthetamonotone}The monotonicity of $z(\theta)$}

The goal of this section is to show that $z(\theta)$ is strictly
increasing on $(\pi/2,\pi)$. We recall from Lemma \ref{lem:zerosdenom}
that the three zeros in $t$ of the polynomial $1+t+at^{2}+z(\theta)t^{2}(t-b)$
are $t_{0}=\tau(\theta)e^{-i\theta}$, $t_{1}=\tau(\theta)e^{i\theta}$,
and $t_{2}=\zeta(\theta)\tau(\theta)$. Consequently 
\begin{equation}
z=-\frac{1+t_{0}+at_{0}^{2}}{t_{0}^{2}(t_{0}-b)}.\label{eq:zfrac}
\end{equation}
If we let $1+t_{0}+at_{0}^{2}=a(t_{0}-\tau_{1})(t_{0}-\tau_{2})$,
then the logarithmic derivatives of both sides and the identity 
\[
dt_{0}=d\tau e^{-i\theta}-i\tau e^{-i\theta}d\theta=t_{0}\left(\frac{d\tau}{\tau}-id\theta\right)
\]
give 
\begin{equation}
\frac{dz}{z}=h(t_{0})\left(\frac{d\tau}{\tau}-id\theta\right),\label{eq:logderiv-z}
\end{equation}
where 
\[
h(t_{0}):=\frac{t_{0}}{t_{0}-\tau_{1}}+\frac{t_{0}}{t_{0}-\tau_{2}}-\frac{t_{0}}{t_{0}-b}-2.
\]
Since $dz/z\in\mathbb{R}$, the imaginary and the real parts of \eqref{eq:logderiv-z}
give
\[
\Im h(t_{0})\frac{d\tau}{\tau}=\Re h(t_{0})d\theta
\]
and 
\begin{align*}
\frac{dz}{z} & =\Re h(t_{0})\frac{d\tau}{\tau}+\Im(h(t_{0}))d\theta.
\end{align*}
We multiply both sides of the second equation by $\Im h(t_{0})$ and
apply the the first equation to obtain 
\begin{equation}
\Im h(t_{0}).\frac{dz}{d\theta}=z\left|h(t_{0})\right|^{2}\label{eq:dzdtheta}
\end{equation}
where 
\begin{align}
\Im h(t_{0}) & =\Im\left(\frac{-\overline{\tau_{1}}t_{0}}{\left|t_{0}-\tau_{1}\right|^{2}}+\frac{-\overline{\tau_{2}}t_{0}}{\left|t_{0}-\tau_{2}\right|^{2}}+\frac{bt_{0}}{\left|t_{0}-b\right|^{2}}\right).\label{eq:Imf-fractions}
\end{align}

\begin{lem}
If $a<0$, then the function $z(\theta)$ is negative and strictly
increasing on $(\pi/2,\pi)$. 
\end{lem}

\begin{proof}
We first note that \eqref{eq:zetatheta} has only positive solutions
in $\zeta$ by Lemma \ref{lem: a<0 roots f*} and consequently $z(\theta)$
is negative by \eqref{eq:zdef} and Lemma \ref{lem:positivetau}.
Since $a<0$, we have $\tau_{1},\tau_{2}\in\mathbb{R}$. With the
identities $\tau_{1}+\tau_{2}=-1/a$ and $\tau_{1}\tau_{2}=1/a$,
we obtain 
\begin{align*}
\tau_{1}\left|t_{0}-\tau_{2}\right|^{2}+\tau_{2}\left|t_{0}-\tau_{1}\right|^{2} & =\tau_{1}(\tau^{2}-2\tau_{2}\tau\cos\theta+\tau_{2}^{2})+\tau_{2}(\tau^{2}-2\tau_{1}\tau\cos\theta+\tau_{1}^{2})\\
 & =-\frac{1}{a}\tau^{2}-\frac{4}{a}\tau\cos\theta-\frac{1}{a^{2}}\\
 & =\frac{1}{a^{2}}\left(-a\tau^{2}-4a\tau\cos\theta-1\right)\\
 & =\frac{1}{a^{2}}\left(-a\tau\left(-1/\zeta-2\cos\theta\right)-4a\tau\cos\theta-1\right)\\
 & =\frac{1}{a^{2}}(a\tau/\zeta-2a\tau\cos\theta-1)<0,
\end{align*}
and the lemma follows from \eqref{eq:dzdtheta}. 
\end{proof}
\begin{lem}
If $0\le a\le1/4$, then the function $z(\theta)$ is negative and
strictly increasing on $(\pi/2,\pi)$. 
\end{lem}

\begin{proof}
From $0\le a\le1/4$, we conclude that $\tau_{1}$ and $\tau_{2}$
are negative and thus 
\begin{align*}
\Im h(t_{0}) & =\frac{\tau_{1}\tau\sin\theta}{\left|t_{0}-\tau_{1}\right|^{2}}+\frac{\tau_{2}\tau\sin\theta}{\left|t_{0}-\tau_{2}\right|^{2}}-\frac{b\tau\sin\theta}{\left|t_{0}-b\right|^{2}}<0.
\end{align*}
Also \eqref{eq:zdef} and Lemmas \ref{lem:asymptotes} and \ref{lem:limitzetaatasymptote0}
imply that $z(\theta)$ is negative on $(\pi/2,\pi)$. The lemma follows
from \eqref{eq:dzdtheta}. 
\end{proof}
We now consider the case $a>1/4$ in which $\tau_{1},\tau_{2}\notin\mathbb{R}$.
If we write $\tau_{1}=x+iy$ and $\tau_{2}=x-iy$, then 
\begin{align}
\Im\left(-\overline{\tau_{1}}t_{0}|t_{0}-\tau_{2}|^{2}-\overline{\tau_{2}}t_{0}|t_{0}-\tau_{1}|^{2}\right) & =\Im\left(-\tau_{2}t_{0}|t_{0}-\tau_{2}|^{2}-\tau_{1}t_{0}|t_{0}-\tau_{1}|^{2}\right)\nonumber \\
 & =2\tau\sin\theta(x\tau^{2}\cos^{2}\theta-2\tau\cos\theta(x^{2}+y^{2})+x(x^{2}+y^{2}+\tau^{2}\sin^{2}\theta))\nonumber \\
 & =2\tau\sin\theta\left(x\tau^{2}-\frac{2}{a}\tau\cos\theta+\frac{x}{a}\right)\nonumber \\
 & =2\tau\sin\theta\left(-\frac{\tau^{2}}{2a}-\frac{2}{a}\tau\cos\theta-\frac{1}{2a^{2}}\right)\nonumber \\
 & =2\tau\sin\theta\left(-\frac{1}{2a}(-1/\zeta-2\cos\theta)^{2}-\frac{2}{a}(-1/\zeta-2\cos\theta)\cos\theta-\frac{1}{2a^{2}}\right)\nonumber \\
 & =2\tau\sin\theta\left(-\frac{1}{2a\zeta^{2}}+\frac{2}{a}\cos^{2}\theta-\frac{1}{2a^{2}}\right)\nonumber \\
 & =\frac{\tau\sin\theta}{a^{2}\zeta^{2}}\left(-a+\zeta^{2}(4a\cos^{2}\theta-1)\right).\label{eq:firsttwoterms}
\end{align}
If $-1/2\sqrt{a}<\cos\theta<0$, then $\Im h(t_{0})<0$. Consequently
Lemmas \ref{lem:asymptotes} and \ref{lem:limitzetaatasymptote0}
and \eqref{eq:dzdtheta} imply that $z(\theta)$ is negative and strictly
increasing on $\left(\pi/2,\cos^{-1}(-1/2\sqrt{a})\right)$.

For the remainder of this section, we will show $z(\theta)$ is strictly
increasing when $a>1/4$ and 
\begin{equation}
-1<\cos\theta<-\frac{1}{2\sqrt{a}}.\label{eq:leftasymp}
\end{equation}
From Lemma \ref{lem:signzeta}, it suffices to show $\Im h(t_{0})>0$.
We first show that \eqref{eq:firsttwoterms} is positive or equivalently
\begin{equation}
\zeta<\frac{-a}{\sqrt{4a^{2}\cos^{2}\theta-a}}.\label{eq:zetainqage1/4}
\end{equation}

Since $\zeta<-1$, this claim is trivial if 
\[
-1<\frac{-a}{\sqrt{4a^{2}\cos^{2}\theta-a}}.
\]
To prove \eqref{eq:zetainqage1/4} for the remaining case, we will
show that the polynomial $f(\zeta,\theta)$ has no zero in $\zeta$
on the interval 
\[
\left[-\frac{a}{\sqrt{4a^{2}\cos^{2}\theta-a}},-1\right]
\]
by showing that this polynomial has one zero on each of the intervals
\begin{equation}
\left(-\infty,-\frac{a}{\sqrt{4a^{2}\cos^{2}\theta-a}}\right),(-1,0),\text{ and }(0,\infty).\label{eq:threeintervals}
\end{equation}
We check the sign of $f(\zeta,\theta)$ at each of the endpoint of
these intervals and apply the Intermediate Value Theorem. We first
note that 
\[
f(0,\theta)=-a<0.
\]
Since the leading coefficient of $f(\zeta,\theta)$ satisfies 
\[
2\cos\theta(1-4a\cos^{2}\theta)>0
\]
by \eqref{eq:leftasymp}, we conclude $\lim_{\zeta\rightarrow-\infty}f(\zeta,\theta)=-\infty$
and $\lim_{\zeta\rightarrow+\infty}f(\zeta,\theta)=+\infty$. From
Lemma \ref{lem:signf*atpm1}, we obtain 
\[
f(-1,\theta)=-f^{*}(-1,\theta)>0.
\]

\begin{lem}
\label{lemmasignoffatfirst pointofinterest} Whenever $a>1/4$ , $b-27a+9>0$,
and \eqref{eq:leftasymp}, we have 
\[
f\left(-\frac{a}{\sqrt{4a^{2}\cos^{2}\theta-a}},\theta\right)>0.
\]
\end{lem}

\begin{proof}
The Cauchy inequality gives 
\[
26a+1/a-7>0.
\]
We expand $f(\zeta,\theta)$ when $\zeta=-a/\sqrt{4a^{2}\cos^{2}\theta-a}$
and collect the terms according to $\zeta$ 
\begin{eqnarray*}
 &  & \frac{4a^{2}\cos^{2}\theta-a}{a^{2}}f\left(-\frac{a}{\sqrt{4a^{2}\cos^{2}\theta-a}},\theta\right)\\
 & = & -12a\cos^{2}\theta+b+4\cos^{2}\theta+1+\frac{a}{\sqrt{4a^{2}\cos^{2}\theta-a}}2\cos\theta(-1+4a\cos^{2}\theta)-\frac{\sqrt{4a^{2}\cos^{2}\theta-a}}{a}2\cos\theta(1-3a)-a\\
 & = & -12a\cos^{2}\theta+b+4\cos^{2}\theta+1+2\cos\theta\sqrt{4a^{2}\cos^{2}\theta-a}-2\cos\theta\sqrt{4a^{2}\cos^{2}\theta-a}(\frac{1}{a}-3)-a\\
 & = & -12a\cos^{2}\theta+b+4\cos^{2}\theta+1+2\cos\theta\sqrt{4a^{2}\cos^{2}\theta-a}(4-\frac{1}{a})-a\\
 & > & -12a\cos^{2}\theta+b+4\frac{1}{4a}+1+8\cos\theta\sqrt{4a^{2}\cos^{2}\theta-a}-2\cos\theta\frac{\sqrt{4a^{2}\cos^{2}\theta-a}}{a}-a\\
 & > & -12a\cos^{2}\theta+b+\frac{1}{a}+1+8\cos\theta\sqrt{4a^{2}\cos^{2}\theta}-a\\
 & = & 4a\cos^{2}\theta+b+\frac{1}{a}+1-a\\
 & > & 4a\frac{1}{4a}+b+\frac{1}{a}+1-a\\
 & = & (b-27a+9)+(26a-7+1/a)>0.
\end{eqnarray*}
\end{proof}
By the Intermediate Value Theorem, $f(\zeta,\theta)$ has a zero on
each of the interval in \eqref{eq:threeintervals} and consequently
it has no zero on 
\[
\left[-\frac{a}{\sqrt{4a^{2}\cos^{2}\theta-a}},-1\right].
\]

Having proved that \eqref{eq:firsttwoterms} is positive, we now show
that the same conclusion holds for $\Im\left(h(t_{0})\right)$. We
multiply both sides of \eqref{eq:Imf-fractions} by $a^{2}|t_{0}-\tau_{1}|^{2}|t_{0}-\tau_{2}|^{2}$
and obtain 
\[
\frac{a^{2}|t_{0}-\tau_{1}|^{2}|t_{0}-\tau_{2}|^{2}}{\tau\sin\theta}\Im(h(t_{0}))=\frac{-a+\zeta^{2}(4a\cos^{2}\theta-1)}{\zeta^{2}}-bz^{2}\tau^{4}.
\]
With \eqref{eq:zdef} and \eqref{eq:taudef}, the right side becomes
\[
\frac{-a+\zeta^{2}(4a\cos^{2}\theta-1)}{\zeta^{2}}-\frac{b}{(1+2\zeta\cos\theta)^{2}}
\]
or 
\[
\frac{\left(-\zeta^{2}+a(2\zeta\cos\theta-1)(2\zeta\cos\theta+1)\right)(1+2\zeta\cos\theta)^{2}-b\zeta^{2}}{\zeta^{2}(1+2\zeta\cos\theta)^{2}}.
\]
Using \eqref{eq:zetatheta}, we replace $-b\zeta^{2}$ by 
\[
\zeta(1+2\zeta\cos\theta)(2\cos\theta+\zeta)-a(1+2\zeta\cos\theta)^{3},
\]
cancel the factor $1+2\zeta\cos\theta$, and collect the terms in
the numerator by $\zeta$ and it remains to show that 
\begin{equation}
G(\zeta):=-2a+2\zeta\cos\theta(1-3a)+\zeta^{3}(2\cos\theta)(4a\cos^{2}\theta-1)>0.\label{eq:Gzeta}
\end{equation}

In the first case when $a\le1/3$, \eqref{eq:zetainqage1/4} implies
that

\[
\zeta^{2}\cos^{2}\theta>\frac{\zeta^{2}+a}{4a}>\frac{1+a}{4a}\geq1
\]
or equivalently $\zeta\cos\theta>1$. With this inequality, \eqref{eq:Gzeta}
follows directly from 
\[
G(\zeta)=2\zeta\cos\theta(1-3a)+2(\zeta^{2}(\zeta\cos\theta)(4a\cos^{2}\theta-1)-a)>0.
\]

On the other hand if $a>1/3$, then we use \eqref{eq:zetatheta} to
solve for $-\zeta^{3}\left(2\cos\theta-8a\cos^{3}\theta\right)$ and
reduce $G(\zeta)$ to a quadratic polynomial in $\zeta$ 
\[
-3a+4\zeta\cos\theta(1-3a)+\zeta^{2}\left(-12a\cos^{2}\theta+4\cos^{2}\theta+b+1\right)
\]
which is at least 
\begin{equation}
-3a+4\zeta\cos\theta(1-3a)+\zeta^{2}(-12a\cos^{2}\theta+4\cos^{2}\theta+27a-8)\label{eq:quadzeta}
\end{equation}
by \eqref{eq:cond}. As a quadratic polynomial in $\zeta$, the value
of \eqref{eq:quadzeta} at $-1$ is 
\[
4(3a-1)(2+\cos\theta-\cos^{2}\theta)>0
\]
and its derivative is 
\begin{align*}
 & 4\cos\theta(1-3a)+2(27a-8+4\cos^{2}\theta-12a\cos^{2}\theta)\zeta\\
= & 6a\zeta+4(1-3a)\left(\cos\theta-2\zeta+2\zeta\left(\cos^{2}\theta-1\right)\right)<0
\end{align*}
when $\zeta<-1$. Thus \eqref{eq:quadzeta} is positive for $\zeta<-1$
and so is $G(\zeta)$.

Having proved that $z(\theta)$ is strictly increasing on $(\pi/2,\pi)$,
we conclude this section with the following lemma.
\begin{lem}
\label{lem:zonto}The function $z(\theta)$ maps $(\pi/2,\pi)$ onto
$I_{a,b}$. 
\end{lem}

\begin{proof}
We will show that the limits of $z(\theta)$ when $\theta$ approaches
$\pi/2$ and $\pi$ give the two endpoints of the interval $I_{a,b}$.
Lemma \ref{lem:limitzetaatasymptote0} and \eqref{eq:taudef} imply
that $\lim_{\theta\rightarrow\pi/2}\tau(\theta)=0$. Thus from \eqref{eq:zfrac}
and the fact that $z(\theta)$ is monotone increasing, we conclude
\[
\lim_{\theta\rightarrow\pi/2}z(\theta)=-\infty.
\]
On the other hand, \eqref{eq:zetapoly} implies that $\lim_{\theta\rightarrow\pi}\zeta(\theta)=\zeta_{0}$
which is the unique zero of 
\[
(8a-2)\zeta^{3}+\zeta^{2}(-12a+b+5)+(6a-2)\zeta-a
\]
 on $(-\infty,-1]\cup[1,\infty)$. The limit 
\[
\lim_{\theta\rightarrow\pi}z(\theta)=\frac{\zeta_{0}^{2}}{(1-2\zeta_{0})^{3}}
\]
follows from \eqref{eq:zdef} and \eqref{eq:taudef}.
\end{proof}

\section{The zeros of $H_{m}(z)$ }

We recall that for each $\theta\in(\pi/2,\pi)$, the functions $\tau(\theta)$
and $z(\theta)$ are defined as in \eqref{eq:taudef} and \eqref{eq:zdef}.
We note that the three zeros $t_{0,1}=\tau(\theta)e^{\pm i\theta}$
and $t_{2}=\zeta(\theta)\tau(\theta)$ of $1+t+at^{2}+zt^{2}(t-b)$
are distinct since they have different arguments. The Cauchy's integral
formula gives 
\begin{align*}
H_{m}(z) & =\frac{1}{2\pi i}\ointctrclockwise_{|t|=\epsilon}\frac{dt}{(1+t+at^{2}+zt^{2}(t-b))t^{m+1}}.
\end{align*}
Since 
\[
\lim_{R\rightarrow\infty}\ointctrclockwise_{|t|=R}\frac{dt}{(1+t+at^{2}+zt^{2}(t-b))t^{m+1}}=0,
\]
we compute the residue of the integrand each distinct zero of $(1+t+at^{2}+zt^{2}(t-b))t^{m+1}$
and obtain 
\[
-zH_{m}(z)=\frac{1}{(t_{0}-t_{1})(t_{0}-t_{2})t_{0}^{m+1}}+\frac{1}{(t_{1}-t_{0})(t_{1}-t_{2})t_{1}^{m+1}}+\frac{1}{(t_{2}-t_{0})(t_{2}-t_{1})t_{2}^{m+1}}.
\]
The reduction of the right side to \eqref{eq:gthetaform} is the same
as that in \cite{tz}, which is provided below for completeness. From
the expression above, we deduce that $z$ is a nonzero root of $H_{m}(z)$
if and only if 
\begin{equation}
\frac{1}{(t_{0}-t_{1})(t_{0}-t_{2})t_{0}^{m+1}}+\frac{1}{(t_{1}-t_{0})(t_{1}-t_{2})t_{1}^{m+1}}+\frac{1}{(t_{2}-t_{0})(t_{2}-t_{1})t_{2}^{m+1}}=0.\label{eq:partialfraczero}
\end{equation}
After multiplying the left side of \eqref{eq:partialfraczero} by
$t_{0}^{m+3}$ we obtain the equality 
\[
\frac{1}{(1-t_{1}/t_{0})(1-t_{2}/t_{0})}+\frac{1}{(t_{1}/t_{0}-1)(t_{1}/t_{0}-t_{2}/t_{0})(t_{1}/t_{0})^{m+1}}+\frac{1}{(t_{2}/t_{0}-1)(t_{2}/t_{0}-t_{1}/t_{0})(t_{2}/t_{0})^{m+1}}=0.
\]
With $\zeta=t_{2}/(t_{0}e^{i\theta})$, we rewrite the left side as
\[
\frac{1}{(1-e^{2i\theta})(1-\zeta e^{i\theta})}+\frac{1}{(e^{2i\theta}-1)(e^{2i\theta}-\zeta e^{i\theta})(e^{2i\theta})^{m+1}}+\frac{1}{(\zeta e^{i\theta}-1)(\zeta e^{i\theta}-e^{2i\theta})(\zeta e^{i\theta})^{m+1}},
\]
or equivalently 
\[
\frac{1}{e^{2i\theta}(-2i\sin\theta)(e^{-i\theta}-\zeta)}+\frac{1}{(2i\sin\theta)(e^{i\theta}-\zeta)(e^{2i\theta})^{m+2}}+\frac{1}{(\zeta-e^{-i\theta})(\zeta-e^{i\theta})(\zeta)^{m+1}(e^{i\theta})^{m+3}}.
\]
We multiply this expression by $(\zeta-e^{-i\theta})(\zeta-e^{i\theta})e^{i(m+3)\theta}$
and set the summation equal to zero to arrive at 
\begin{align}
0 & =\frac{(\zeta-e^{i\theta})e^{i(m+1)\theta}}{2i\sin\theta}+\frac{e^{-i\theta}-\zeta}{(2i\sin\theta)e^{i(m+1)\theta}}+\frac{1}{\zeta^{m+1}}\nonumber \\
 & =\frac{(\zeta-e^{i\theta})e^{i(m+1)\theta}-(\zeta-e^{-i\theta})e^{-i(m+1)\theta}}{2i\sin\theta}+\frac{1}{\zeta^{m+1}}\nonumber \\
 & =\frac{\zeta(e^{i(m+1)\theta}-e^{-i(m+1)\theta})+e^{-i(m+2)\theta}-e^{i(m+2)\theta}}{2i\sin\theta}+\frac{1}{\zeta^{m+1}}\nonumber \\
 & =\frac{\zeta(2i\sin\left((m+1)\theta\right))-2i\sin\left((m+2)\theta\right)}{2i\sin\theta}+\frac{1}{\zeta^{m+1}}\nonumber \\
 & =\frac{2i\zeta\sin\left((m+1)\theta\right)-2i\sin\left((m+1)\theta\right)\cos\theta-2i\cos\left((m+1)\theta\right)\sin\theta}{2i\sin\theta}+\frac{1}{\zeta^{m+1}}\nonumber \\
 & =\frac{(\zeta-\cos\theta)\sin\left((m+1)\theta\right)}{\sin\theta}-\cos\left((m+1)\theta\right)+\frac{1}{\zeta^{m+1}}.\label{eq:gthetaform}
\end{align}
We define the function $g_{m}(\theta)$ on $(\pi/2,\pi)$ as in \eqref{eq:gthetaform}.
By Lemma \ref{lem:asymptotes}, $g_{m}(\theta)$ has a vertical asymptote
at $\cos^{-1}(-1/(2\sqrt{a}))$ if $a>1/4$. 
\begin{lem}
\label{lem:singinterval} Suppose $1/4<a$ and $m\ge6$. Let $J_{h}\subset(\pi/2,\pi)$
be the interval 
\begin{equation}
\begin{cases}
\left(\frac{h-1}{m+1}\pi,\frac{h}{m+1}\pi\right) & \text{ if }\left\lfloor (m+1)/2\right\rfloor +2\le h\le m+1\\
\left(\frac{\pi}{2},\frac{h}{m+1}\pi\right) & \text{\text{ if }}h=\left\lfloor (m+1)/2\right\rfloor +1.
\end{cases}\label{eq:subintdef}
\end{equation}
If 
\[
\cos^{-1}\left(-\frac{1}{2\sqrt{a}}\right)\in J_{h},
\]
then $g(\theta)$ has at least two zeros in $J_{h}$ whenever $\left\lfloor (m+1)/2\right\rfloor +2\le h\le m$,
and at least one zero whenever $h=m+1$ or $h=\left\lfloor (m+1)/2\right\rfloor +1$. 
\end{lem}

\begin{proof}
The vertical asymptote of $g_{m}(\theta)$ at $\cos^{-1}(-1/2\sqrt{a})$
divides the interval $J_{h}$ in \eqref{eq:subintdef} into two subintervals.
We will show that each subinterval contains at least one zero of $g_{m}(\theta)$
if $\left\lfloor (m+1)/2\right\rfloor +2\le h\le m$. In the case
$h=m+1$, the subinterval on the left of the asymptote contains at
least one zero of $g_{m}(\theta)$. On the other hand if $h=\left\lfloor (m+1)/2\right\rfloor +1$,
then the subinterval on the right contains at least a zero of $g_{m}(\theta)$.
We analyze these two subintervals in the two cases below.

We consider the first case when $\theta\in J_{h}$ and $\theta<\cos^{-1}(-1/2\sqrt{a})$.
From \eqref{eq:gthetaform} and the inequality $\left|\zeta(\theta)\right|>1$,
we see that the sign of $g_{m}(\theta)$ at the left-end point of
$J_{h}$, for $\left\lfloor (m+1)/2\right\rfloor +2\le h\le m+1$,
is $(-1)^{h}$. We now show that the sign of $g_{m}(\theta)$ is $(-1)^{h-1}$
when $\theta\rightarrow\cos^{-1}(-1/2\sqrt{a})$. From Lemmas \ref{lem:asymptotes}
and \ref{lem:limitzetaatasymptote0}, we observe that $\zeta(\theta)\rightarrow+\infty$
as $\theta\rightarrow\cos^{-1}(-1/2\sqrt{a})$. Since $\theta\in J_{h}$,
the sign of $\sin\left((m+1)\theta\right)$ is $(-1)^{h-1}$ and consequently
the sign of $g_{m}(\theta)$ is $(-1)^{h-1}$ when $\theta\rightarrow\cos^{-1}(-1/2\sqrt{a})$
by \eqref{eq:gthetaform}. By the Intermediate Value Theorem, we obtain
at least one zero of $g_{m}(\theta)$ in this case.

Next we consider the case when $\theta\in J_{h}$ and $\theta>\cos^{-1}(-1/2\sqrt{a})$.
In this case the sign of $g_{m}(\theta)$ at the right-end point of
$J_{h}$, for $\left\lfloor (m+1)/2\right\rfloor +1\le h\le m$, is
$(-1)^{h-1}$ . Since $\zeta(\theta)\rightarrow-\infty$ as $\theta\rightarrow\cos^{-1}(-1/2\sqrt{a})$
by Lemma \ref{lem:signzeta} and the sign of $\sin\left((m+1)\theta\right)$
is $(-1)^{h-1}$, the sign of $g_{m}(\theta)$ is $(-1)^{h}$ as $\theta\rightarrow\cos^{-1}(-1/2\sqrt{a})$
and we obtain at least one zero of $g_{m}(\theta)$ by the Intermediate
Value Theorem. 
\end{proof}
We note that Lemma \ref{lem:densesubint} allows us to ignore the
case when an endpoint of $J_{h}$ coincides with $\cos^{-1}(-1/(2\sqrt{a}))$. 
\begin{lem}
\label{lem:signend}If $a<1/4$, then the sign of $g_{m}(\pi^{-})$
is $(-1)^{m}$. 
\end{lem}

\begin{proof}
As $\theta\rightarrow\pi$, the leading coefficient of $f(\zeta,\theta)$
approaches $-2+8a<0$ and $f(1,\theta)$ approaches $1+a+b\ge0$ .
Thus $f(\zeta,\theta)$ has a solution on $(1,\infty)$ when $\theta$
is close to $\pi$ and consequently $\zeta(\theta)>1$ by the definition
of $\zeta(\theta)$ in Section \ref{sec:auxiliary-functions}. The
result follows directly from \eqref{eq:gthetaform} and the fact that
\[
\lim_{\theta\to\pi^{-}}\frac{\sin\left((m+1)\theta\right)}{\sin(\theta)}=(m+1)(-1)^{m}.
\]
\end{proof}
With all the lemmas at our disposal, we now prove the sufficient condition
of Theorem \ref{thm:maintheorem} for the two cases $a\le1/4$ and
$a>1/4$. In the first case, Lemma \ref{lem:asymptotes} shows that
the function $\zeta(\theta)$ is continuous on $(\pi/2,\pi)$. From
the formula of $g_{m}(\theta)$ in \eqref{eq:gthetaform} and Lemma
\ref{lem:signend}, this function changes its sign at the endpoints
of $J_{h}$, $\left\lfloor (m+1)/2\right\rfloor +2\le h\le m+1$,
in \eqref{eq:subintdef} and thus it has at least 
\[
m-\left\lfloor (m+1)/2\right\rfloor =\left\lfloor m/2\right\rfloor 
\]
zeros on $(\pi/2,\pi)$. Each such zero gives us a real zero of $H_{m}(z)$
by the monotone map $z(\theta)$ and the reality of the zeros of $H_{m}(z)$
follows from Lemma \ref{lem:degreeHm} and the Fundamental Theorem
of Algebra. On the other hand, if $a>1/4$, then we obtain at least
$\left\lfloor m/2\right\rfloor -1$ of $g_{m}(\theta)$ on the intervals
$J_{h}$, $\left\lfloor (m+1)/2\right\rfloor +2\le h\le m$ by the
same argument. By Lemma \ref{lem:singinterval}, the interval $J_{h}$
containing the vertical asymptote $\cos^{-1}\left(-1/2\sqrt{a}\right)$
gives us another zero of $g_{m}(\theta)$ and we conclude all the
zeros of $H_{m}(z)$ are real. For the density of these zeros, the
Intermediate Value Theorem shows that $\bigcup_{m=0}^{\infty}\mathcal{Z}(g_{m}(\theta))$
is dense on $(\pi/2,\pi)$ . From Lemma \ref{lem:zonto}, we conclude
that $\bigcup_{m=0}^{\infty}\mathcal{Z}(H_{m})$ is dense on $I_{a,b}$
since the map $z(\theta)$ is continuous. 

\section{The necessary condition for the reality of the zeros of $H_{m}(z)$}

The initial setup to prove the necessary condition is similar to that
in \cite{tz}. For completeness, we quickly review this setup and
then focus on the key differences starting from Lemma \ref{lem:nonrealz}.
We recall some definitions (from \cite{sokal}) related to the root
distribution of a sequence of functions 
\[
f_{m}(z)=\sum_{k=1}^{n}\alpha_{k}(z)\beta_{k}(z)^{m},
\]
where $\alpha_{k}(z)$ and $\beta_{k}(z)$ are analytic in a domain
$D$. We say that an index $k$ is dominant at $z$ if $|\beta_{k}(z)|\ge|\beta_{l}(z)|$
for all $l$ ($1\le l\le n$). Let 
\[
D_{k}=\{z\in D:k\mbox{ is dominant at }z\}.
\]
Let $\liminf\mathcal{Z}(f_{m})$ be the set of all $z\in D$ such
that every neighborhood $U$ of $z$ has a non-empty intersection
with all but finitely many of the sets $\mathcal{Z}(f_{m})$. Let
$\limsup\mathcal{Z}(f_{m})$ be the set of all $z\in D$ such that
every neighborhood $U$ of $z$ has a non-empty intersection with
infinitely many of the sets $\mathcal{Z}(f_{m})$. The necessary condition
for the reality of zeros of $H_{m}(z)$ relies on following theorem
from Sokal (\cite[Theorem 1.5]{sokal}). 
\begin{thm}
\label{sokal}Let $D$ be a domain in $\mathbb{C}$, and let $\alpha_{1},\ldots,\alpha_{n},\beta_{1},\ldots,\beta_{n}$
$(n\ge2)$ be analytic functions on $D$, none of which is identically
zero. Let us further assume a 'no-degenerate-dominance' condition:
there do not exist indices $k\ne k'$ such that $\beta_{k}\equiv\omega\beta_{k'}$
for some constant $\omega$ with $|\omega|=1$ and such that $D_{k}$
$(=D_{k'})$ has nonempty interior. For each integer $m\ge0$, define
$f_{m}$ by 
\[
f_{m}(z)=\sum_{k=1}^{n}\alpha_{k}(z)\beta_{k}(z)^{m}.
\]
Then $\liminf Z(f_{m})=\limsup Z(f_{m})$, and a point $z$ lies in
this set if and only if either

(i) there is a unique dominant index $k$ at $z$, and $\alpha_{k}(z)=0$,
or

(ii) there a two or more dominant indices at $z$. 
\end{thm}

Using \eqref{eq:partialfraczero}, we apply Theorem \ref{sokal} with
\[
\alpha_{k}(z)=\frac{1}{t_{k}}\prod_{i\ne k}\frac{1}{(t_{i}-t_{k})}\qquad\text{and}\qquad\beta_{k}(z)=\frac{1}{t_{k}}
\]
and deduce that $z\in\liminf\mathcal{Z}(H_{m})=\limsup\mathcal{Z}(H_{m})$
if and only if the two smallest (in modulus) zeros of $P(t)+zQ(t)$
have the same modulus. Thus if we can find $z\notin\mathbb{R}$ with
this property then for large $m$, not all the zeros of $H_{m}(z)$
are real by the definition of $\liminf\mathcal{Z}(H_{m})$. The following
lemma shows it is sufficient to find a suitable $\zeta$. 
\begin{lem}
\label{lem:nonrealz}Assume $\cos\theta\neq0$. If $1/\zeta$ is a
nonreal solution of $f^{*}$ such that $\left|1/\zeta\right|<1$,
then $\zeta\tau^{3}\notin\mathbb{R}$. 
\end{lem}

\begin{proof}
Since $\Arg(\zeta)=-\Arg(1/\zeta)$ and $|\zeta|\ne|1/\zeta|$, we
conclude that 
\[
\zeta+1/\zeta\notin\mathbb{R}.
\]
As a consequence, \eqref{eq:taudef} gives 
\[
\tau(2\cos\theta+\zeta)=-\left(\frac{1}{\zeta}+2\cos\theta\right)(2\cos\theta+\zeta)
\]
which is nonreal after we expand the product. The lemma follows from
\eqref{eq:tauequation}.
\end{proof}
From \eqref{eq:zdef} and Lemmas \ref{lem:zerosdenom} and \ref{lem:nonrealz},
it suffices to find $\theta^{*}\ne\pi/2$ such that $f^{*}$ has a
solution $\zeta^{*}\notin\mathbb{R}$ with $|\zeta^{*}|<1$. We will
find such a $\theta^{*}$ for the two cases $a<-b-1$ and $a>(b+9)/27$.

\subsection*{Case $a<-b-1$}

From \eqref{eq:zetapolyrecip}, we observe that the roots in $\zeta$
of $f^{*}(\zeta,\pi/2)$ are $0,\pm i\sqrt{-(1+b)/a}$. The inequalities
$a<-b-1<0$ imply that there is $\theta^{*}$ sufficiently close to
$\pi/2$ so that $f^{*}$ has a nonreal root inside the open unit
disk. 

\subsection*{Case $a>(b+9)/27$}

We first note that the discriminant of $f^{*}(\zeta,\theta)$ as a
cubic polynomial in $\zeta$ is a polynomial in $\cos^{2}\theta=:x$,
which is denoted by $\Delta(x)$. Computer algebra shows that the
discriminant of $\Delta(x)$ in $x$ is 
\[
-65536b\left(27a^{2}b-9ab+b+1\right)^{3}\left(ab^{2}+b+1\right)<0
\]
and thus $\Delta(x)$ has a unique real zero denoted by $x'$. Since
\begin{align}
\Delta(0) & =4a(b+1)^{3}>0,\label{eq:hzero}\\
\Delta(1) & =-4(27a-b-9)\left(ab^{2}+b+1\right)<0,\label{eq:hone}
\end{align}
we have $0<x'<1$. By the definition of $x'$, the polynomial $f^{*}(\zeta,\cos^{-1}\sqrt{x'})$
has a multiple zero which is denoted by $\zeta'$. 

We will show later that $|\zeta'|<1$. Assuming this inequality, we
choose $\sqrt{x'}<\cos\theta^{*}\ll1$. From \eqref{eq:hzero} and
\eqref{eq:hone}, we conclude that the discriminant $f^{*}(\zeta,\theta^{*})$
is negative. Since $\zeta'$ is a multiple zero of $f^{*}(\zeta,\cos^{-1}\sqrt{x'})$
and $f^{*}(\zeta,\theta^{*})$ has only one real zero, the inequality
$|\zeta'|<1$ implies that for $\cos\theta^{*}$ sufficiently close
to $\sqrt{x'}$, $f^{*}(\zeta,\theta^{*})$ has a non-real zero inside
the open unit disk.

For the reminder of this case, we prove $|\zeta'|<1$. We note that
$\zeta'$ is the zero of the remainder of the polynomial division
of the cubic polynomial $f^{*}(\zeta,\cos^{-1}\sqrt{x'})$ and its
derivative. Since this remainder is linear in $\zeta$, we can easily
solve for $\zeta'$ from a computer
\begin{equation}
\zeta'=-\frac{\sqrt{x'}\left(-3a\left(b+8x'-2\right)+b+4x'+1\right)}{3a\left(b-4x'+1\right)+4x'}.\label{eq:zetazero}
\end{equation}
Next, the rational function 
\[
r(x):=\frac{1+6a+b-3ab+(4-24a)x}{3a+3ab+(4-12a)x}
\]
is decreasing because its derivative 
\[
\frac{dr}{dx}=-\frac{4\left(27a^{2}b-9ab+b+1\right)}{(3ab-12ax+3a+4x)^{2}}<0
\]
since 
\[
1+b-9ab+27a^{2}b>1+b+ab^{2}>0
\]
where we apply the inequality $27a>b+9$ to the first expression.
We note that the inequality above also implies that the numerator
and the denominator of \eqref{eq:zetazero} cannot be both zero since
\[
\frac{1+6a+b-3ab}{24a-4}\ne\frac{3a+3ab}{12a-4}.
\]
We also have
\begin{align*}
r\left(\frac{9a+b+1}{4(9a-2)}\right) & =-1,\\
r\left(\frac{-6ab+3a+b+1}{12a}\right) & =1.
\end{align*}
Next, we show that $r(x)$ is continuous on 
\begin{equation}
\left(\frac{-6ab+3a+b+1}{12a},\frac{9a+b+1}{4(9a-2)}\right)\label{eq:critint}
\end{equation}
by showing that the vertical asymptote of $r(x)$ is outside this
interval. Indeed, we have 
\begin{equation}
\frac{3a+3ab}{12a-4}>\frac{9a+b+1}{4(9a-2)}>\frac{-6ab+3a+b+1}{12a}\label{eq:asymprxloc}
\end{equation}
since the difference of the first two terms and the last two terms
are 
\[
\frac{27a^{2}b-9ab+b+1}{4(3a-1)(9a-2)}>0
\]
and 
\[
\frac{27a^{2}b-9ab+b+1}{12a(3a-1)}>0
\]
respectively. As a consequence $|r(x)|<1$ for all $x$ in \eqref{eq:critint}.
From \eqref{eq:zetazero}, if $x'$ is in this interval, then $|\zeta'|<1$.

On the other hand, if $x'$ does not belong to this interval, then
the inequalities $\Delta(0)>0$ and

\[
\Delta\left(\frac{9a+b+1}{4(9a-2)}\right)=-\frac{(27a-b-9)\left(27a^{2}b-9ab+b+1\right)^{2}}{(9a-2)^{3}}<0,
\]
and the Intermediate Value Theorem imply that 
\begin{equation}
0<x'<\frac{-6ab+3a+b+1}{12a}<\frac{9a+b+1}{4(9a-2)}\label{eq:xzeroupperbound}
\end{equation}
and
\[
\Delta\left(\frac{-6ab+3a+b+1}{12a}\right)=\frac{(2ab-a+b+1)\left(27a^{2}b-9ab+b+1\right)^{2}}{27a^{3}}<0.
\]
We note that the first inequality implies $r(x')>1$ and the second
inequality implies $b<1/2$. From \eqref{eq:zetazero} and \eqref{eq:xzeroupperbound},
to prove $|\zeta'|<1$, it suffices to show 
\[
r(x')<2\sqrt{\frac{9a-2}{9a+b+1}}.
\]
By the monotonicity and continuity of $r(x)$ given in \eqref{eq:asymprxloc},
this inequality is equivalent to 
\[
x'>r^{-1}\left(2\sqrt{\frac{9a-2}{9a+b+1}}\right)
\]
where, with a computer algebra, the right side is 
\[
-\frac{-3ab+(-6ab-6a)\sqrt{\frac{9a-2}{9a+b+1}}+6a+b+1}{4(6a-2)\sqrt{\frac{9a-2}{9a+b+1}}+4(1-6a)}.
\]
By the Intermediate Value Theorem applied to $\Delta(x)$, it remains
to prove at the value $x$ above, $\Delta(x)>0$. By a computer algebra,
such value of $\Delta(x)$ is 
\[
-\frac{4\left(27a^{2}b-9ab+b+1\right)^{2}\left(\left(18a^{2}b+72a^{2}+2ab-16a-2b-2\right)\left(\sqrt{\frac{9a-2}{9a+b+1}}-1\right)-3+12a-3b+16ab+ab^{2}\right)}{(9a+b+1)\left((6a-2)\sqrt{\frac{9a-2}{9a+b+1}}-6a+1\right)^{3}}.
\]
The denominator of the expression above is negative since 
\[
(6a-2)\sqrt{\frac{9a-2}{9a+b+1}}-6a+1<(6a-2)-6a+1<0.
\]
To show 
\[
\left(18a^{2}b+72a^{2}+2ab-16a-2b-2\right)\left(\sqrt{\frac{9a-2}{9a+b+1}}-1\right)-3+12a-3b+16ab+ab^{2}>0
\]
we need to show 
\[
\left(1-\frac{-3+12a-3b+16ab+ab^{2}}{18a^{2}b+72a^{2}+2ab-16a-2b-2}\right)^{2}-\frac{9a-2}{9a+b+1}<0.
\]
Note that the left side is 
\[
-\frac{(27a-b-9)\left(108a^{3}b^{2}+432a^{3}b+a^{2}b^{4}+20a^{2}b^{3}+64a^{2}b^{2}-36a^{2}b-2ab^{3}-14ab^{2}-4ab+8a+b^{2}+2b+1\right)}{4(9a+b+1)\left(9a^{2}b+36a^{2}+ab-8a-b-1\right)^{2}}.
\]
We apply the inequalities $a\ge1/3$ and $0\le b<1/2$ to conclude
that the four differences $432a^{3}b-36a^{2}b$, $20a^{2}b^{3}-2ab^{3}$,
$64a^{2}b^{2}-14ab^{2}$, and $8a-4ab$ are nongegative. Thus 
\[
108a^{3}b^{2}+432a^{3}b+a^{2}b^{4}+20a^{2}b^{3}+64a^{2}b^{2}-36a^{2}b-2ab^{3}-14ab^{2}-4ab+8a+b^{2}+2b+1>0
\]
and we complete this case. 

We end the paper by proposing the following problem.
\begin{problem}
Characterize all linear polynomials $A(z)$, $B(z)$, and $C(z)$
such that the zeros of $P_{m}(z)$ defined in \eqref{eq:fourtermrecurrence}
are real for all $m$. 
\end{problem}

\end{document}